\documentclass[12pt,letterpaper]{article}

\usepackage{amsfonts,amssymb,latexsym,amsthm, amsmath,color,tikz}

\newcommand{\Li}{\mathrm{Li}\,}
\newcommand{\Ls}{\mathrm{Ls}\,}
\newcommand{\Lim}{\mathrm{Lim}\,}
\newcommand{\gbw}{\mathrm{GBW}}

\newtheorem{theorem}{Theorem}[section]
\newtheorem{lemma}[theorem]{Lemma}
\newtheorem{corollary}[theorem]{Corollary}
\newtheorem{question}[theorem]{Question}
\newtheorem{example}[theorem]{Example}

\theoremstyle{definition}

\begin{document}

\author{Ramiro de la Vega\thanks{Universidad de los Andes, Bogot\'a, Colombia, rade@uniandes.edu.co}}

\title{The generalized Bolzano-Weierstrass property revisited}

\maketitle

\begin{abstract}
We investigate the question of when a topological space $X$ has the \emph{Generalized Bolzano-Weierstrass property}: every sequence of subsets of $X$ has a convergent subsequence (in the sense of Kuratowski).
\end{abstract}



\section{Introduction}

For a topological space $X$ and a sequence $(K_n)_{n \in \omega}$ of non-empty subsets of $X$ we define the \emph{lower closed limit} of the sequence as the set $$\Li (K_n)_{n \in \omega}=\left\{ x\in X\ \middle \vert \begin{array}{l}	\text{for every neighborhood } U \text{ of }x\\ U\cap K_n\neq \emptyset \text{ for all but finite }n \end{array}\right\}$$ and define the \emph{upper closed limit} as $$\Ls (K_n)_{n \in \omega}=\left\{ x\in X\ \middle \vert \begin{array}{l}	\text{for every neighborhood } U \text{ of }x\\ U\cap K_n\neq \emptyset \text{ for infinitely many }n \end{array}\right\}.$$

Clearly $\Li (K_n)_{n \in \omega} \subseteq \Ls (K_n)_{n \in \omega}$ for any sequence and when equality occurs we say that the sequence \emph{converges to} this common value $\Lim (K_n)_{n \in \omega} = \Li (K_n)_{n \in \omega} = \Ls (K_n)_{n \in \omega}$. 

This type of convergence is often called \emph{Kuratowski convergence} and its main properties for the metric case can be found in \cite[\S29]{Kur} including a proof of the following result (attributed to Zarankiewicz \cite{Zar}) that Kuratowski called \emph{Generalized Bolzano-Weierstrass theorem}:

\begin{theorem}\label{gbwt}
	Every sequence of subsets of a separable metric space contains a convergent subsequence.
\end{theorem}

This same theorem was called \emph{Selection Theorem} by Hausdorff (see \cite[\S28]{Hau}). In \cite{Sie}, Sierpi\'nski proves a converse to this theorem under the continuum hypothesis:

\begin{theorem}[Assume $2^{\aleph_0}=\aleph_1$]\label{sierpinski}
	A metric space in which every sequence of sets has a convergent subsequence is separable.
\end{theorem}

The purpose of this note is to examine the role of CH in Sierpi\'nski's theorem and, more importantly, to investigate the generalized Bolzano-Weierstrass property in the context of general topological spaces.

\section{Relation with the splitting number $\mathfrak{s}$}

Remember that an $\mathcal{S}\subseteq [\omega]^\omega$ is called a \emph{splitting family} if for any $A \in [\omega]^\omega$ there is an $S \in \mathcal{S}$ such that $A \cap S$ and $A \setminus S$ are both infinite. The \emph{splitting number} $\mathfrak{s}$ is defined as the smallest size of a splitting family. The splitting number is attributed to Booth who used it in \cite{Boo} to study sequential compactness. Using that the Cantor cube $2^\kappa$ is sequentially compact for $\kappa < \mathfrak{s}$, Booth proved the following theorem for which we include a more direct proof.

\begin{theorem}\label{booth}
	If $X$ is a topological space such that $w(X)<\mathfrak{s}$, then
	every sequence of subsets of $X$ has a convergent subsequence.
\end{theorem}
\begin{proof}
	Suppose that there is a sequence $(K_n)_{n \in \omega}$ of
	subsets of $X$ without a convergent subsequence. We want to show
	that $|w(X)|\geq \mathfrak{s}$. For this lets fix a base $\mathcal{B}$ for the topology of $X$ and for each $U \in \mathcal{B}$ let
	$A_U= \{n \in \omega : K_n \cap U=\emptyset\}$. It is enough to
	show that $\{A_U : U \in \mathcal{B}\}$ is a splitting family of
	subsets of $\omega$. Given $A \in [\omega]^\omega$ we know
	that the subsequence $(K_n)_{n \in A}$ does not converge and
	therefore there is a $p \in \Ls(K_n)_{n \in A} \setminus
	\Li(K_n)_{n \in A}$. Since $p \notin \Li(K_n)_{n \in A}$, there
	exists $U \in \mathcal{B}$, a neighborhood of $p$, such that $A_U \cap A$ is infinite.
	Since $p \in \Ls(K_n)_{n \in A}$, the set $A \setminus A_U$ must also be
	infinite, and therefore $A_U$ splits $A$.
\end{proof}

As an immediate consequence, since the weight of a metrizable space is equal to its density, we obtain an improvement of Theorem \ref{gbwt}:

\begin{corollary}\label{gbw for metric}
	If $X$ is a metrizable space such that $d(X)<\mathfrak{s}$, then every
	sequence of subsets of $X$ has a convergent subsequence.
\end{corollary}

This already tells us that we cannot simply remove $CH$ from Theorem \ref{sierpinski} since the discrete space of size $\aleph_1$ would be a consistent counterexample. On the other hand, we can modify Sierpi\'nski's original argument in \cite{Sie} to obtain the following

\begin{theorem}\label{discrete} Suppose that $X$ is a topological space which has a discrete subspace $D$ with $|D|\geq \mathfrak{s}$. Then there is a sequence $(K_n)_{n\in\omega}$ of subsets of $X$ with no convergent subsequence.
\end{theorem}
\begin{proof}
	By assumption, there is an inyective function $\varphi:\mathcal{S}
	\to D$ with $\mathcal{S}$ a splitting family of subsets of
	$\omega$. For each $n \in \omega$ we define a subset $K_n$ of $D$
	by $K_n=\{\varphi(S) : S \in \mathcal{S} \text{ and } n \in S\}$.
	Note that the inyectivity of $\varphi$ ensures that for all $n \in
	\omega$ and all $S \in \mathcal{S}$ we have: $\varphi(S) \in K_n
	\iff n \in S$.
	
	To see that $(K_n)_{n\in\omega}$ has no convergent subsequence,
	fix an infinite $A \subseteq \omega$. Let $S \in \mathcal{S}$ be
	such that $A \cap S$ and $A \setminus S$ are both infinite and let
	$p=\varphi(S)$. Now it is enough to show that $p \in \Ls (K_n)_{n \in A} \setminus \Li (K_n)_{n \in A}$. Since $D$ is discrete,
	there is an open $U \subseteq X$ such that $U \cap D=\{p \}$.
	Thus, for any $n \in A \setminus S$ we have that $U \cap
	K_n=\emptyset$ and therefore $p \notin \Li(K_n)_{n \in A}$. On
	the other hand, given any neighborhood $V$ of $p$ we have that $p
	\in V \cap K_n$ (and hence $V \cap K_n \neq \emptyset$) for every
	$n \in A \cap S$, so that $p \in \Ls (K_n)_{n \in A}$.
\end{proof}

From the fact that any metrizable space $X$ contains a discrete subset of size $d(X)$ (see for example \cite[Theorem 8.1]{Hod}), we now get an improvement of Theorem \ref{sierpinski}.

\begin{corollary} If $X$ is a metrizable space in which every sequence of subsets has a convergent subsequence, then $d(X)<\mathfrak{s}$.
\end{corollary}

Note that discrete spaces (and their one-point compactifications) show that the number $\mathfrak{s}$ is optimal in both Theorem \ref{booth} and Theorem \ref{discrete} (even for compact Hausdorff spaces).

\section{Two well known spaces}

Let us say that a space $X$ has the \emph{generalized Bolzano-Weierstrass property} or that $X \in \gbw$ if every sequence of subsets of $X$ has a convergent subsequence. In the previous section we showed that $X \in \gbw \Leftrightarrow d(X)<\mathfrak{s}$ for metrizable spaces. We now turn to the natural question of whether either of the theorems \ref{booth} or \ref{discrete} characterize $\gbw$ for general topological spaces.

Our first example shows that the converse of Theorem \ref{booth} is not true, since the Sorgenfrey line has weight $2^{\aleph_0} \geq \mathfrak{s}$. It also shows that $\gbw$ is not closed under products since the square of the Sorgenfrey line (a.k.a. The Sorgenfrey plane) contains a discrete subset of size continuum and hence, by Theorem \ref{discrete}, it is not in $\gbw$.

\begin{example}\label{sorgenfrey}
	The Sorgenfrey line belongs to $\gbw$.
\end{example}

\begin{proof}
	Fix a sequence $(K_n)_{n \in \omega}$ of subsets of	$\mathbb{R}$. By Theorem \ref{booth} we may assume that $(K_n)_{n \in \omega}$ converges (say to $K$) in the usual topology of $\mathbb{R}$. Lets momentarily say that a point $x \in K$ and a subset $A \subseteq \omega$ form a \emph{bad pair} $(x,A)$ if there is an $r>x$ such that the set $\{n \in A : [x,r) \cap
	K_n=\emptyset \}$ is infinite. Note that if $(x,A)$ is a bad pair and $A \subseteq^* B$ then $(x,B)$ is also a bad pair. Also note that if $(x,\omega)$ is a bad pair and $r>x$ witnesses this fact, then $(x,r) \cap K=\emptyset$ since $(K_n)_{n \in \omega}$ converges to $K$ with respect to the usual topology of $\mathbb{R}$. It follows that there are only countably many $x \in K$ which can be part of a bad pair, say $\{x_i\}_{i\in \omega}$. 
	
	Now we define inductively a sequence $\{A_i\}_{i\in \omega}$ of subsets of $\omega$: let $A_0=\omega$ and suppose we already defined $A_i$. If $(x_i,A_i)$	is a bad pair we choose $r_i>x_i$ such that $A_{i+1}:=\{n \in A_i	: [x_i,r_i) \cap K_n=\emptyset \}$ is infinite. Otherwise we let $A_{i+1}=A_i$. Finally let $A$ be an infinite subset of $\omega$ with $A \subseteq^* A_i$ for all $i \in \omega$.
	
	We show now that the subsequence $(K_n)_{n \in A}$ is convergent in	the Sorgenfrey topology. Note that $\Ls(K_n)_{n \in A} \subseteq K$ since the Sorgenfrey topology is finner than the usual topology. If $x \in K \setminus	\Li (K_n)_{n \in A}$, then $(x,A)$ is a bad pair and therefore $x=x_i$ for some $i \in \omega$. Since $A \subseteq^* A_i$ we have that $(x_i,A_i)$ is also a bad pair and by construction $r_i$ is defined and $[x_i,r_i) \cap K_n=\emptyset$ for all $n \in A_{i+1}$ (and hence for almost all $n \in A$), so $x \notin \Ls(K_n)_{n \in	A}$.
\end{proof}

Answering the question for Theorem \ref{discrete} seems to be a bit more subtle. Before giving our second example lets make a simple observation.

\begin{lemma}
	If $X\in \gbw$ is countably compact then $X$ is sequentially compact.
\end{lemma}
\begin{proof}
	Fix a sequence $(x_n)_{n\in \omega}$ in $X$. Since $X \in \gbw$, the sequence $(\{x_n\})_{n\in \omega}$ has a convergent subsequence, say $(\{x_n\})_{n\in A}$. Countable compactness gives us that $\Ls(\{x_n\})_{n\in A} \neq \emptyset$ and therefore $\Li(\{x_n\})_{n\in A}\neq \emptyset$. Now for any $x \in \Li(\{x_n\})_{n\in A}$ we have that  $(x_n)_{n\in A}$ converges to $x$ (of course, if $X$ is also Hausdorff, there would be only one such $x$).
\end{proof}

Our second example is only a consistent one. It was constructed by Fedor\v{c}uk in \cite{Fed} under Jensen´s $\diamondsuit$ and it is a hereditarily separable, compact Hausdorff space with no non-trivial convergent sequences. This example shows that the converse of Theorem \ref{discrete} is not a theorem of ZFC since a discrete subset of a hereditarily separable space must be countable.

\begin{example}[Assume $\diamondsuit$]\label{fedorchuk}
	The Fedor\v{c}uk´s space does not belong to $\gbw$.
\end{example}
\begin{proof}
	Being an infinite space with no non-trivial convergent sequences, Fedor\v{c}uk space is not sequentially compact. Since it is compact, it cannot be in $\gbw$ by the previous lemma.
\end{proof}

We still don´t know the answer to the following

\begin{question}\label{Qds}
	Is it consistent with ZFC that any topological space $X \not\in \gbw$ contains a discrete subspace of size $\mathfrak{s}$?
\end{question}

\section{Relation with the tower number $\mathfrak{t}$}

Remember that a $\mathcal{T}\subseteq [\omega]^\omega$ is called a \emph{tower} if it is well-ordered by $\supseteq^*$ and there is no $A\in [\omega]^\omega$ which is almost contained in every member of $\mathcal{T}$. The tower number $\mathfrak{t}$ is defined as the smallest size of a tower. The \emph{hereditary Lindel\"{o}f degree} of a space $X$ is the least cardinal $\kappa$ such that any open cover of a subspace of $X$ has a subcover of size at most $\kappa$. It is well known (see \cite{Haj}) that $hL(X)$ is also the supremum of all $\kappa$ for which there is in $X$ a right-separated sequence of length $\kappa$ (i.e. a sequence $(x_\alpha)_{\alpha \in \kappa}$ such that $x_\alpha \not\in \overline{\{x_\beta: \beta > \alpha\}}$ for all $\alpha \in \kappa$).

Since for any space $X$ we have $hL(X) \leq w(X)$ and since $\mathfrak{t}\leq\mathfrak{s}$, the following result is not comparable with Theorem \ref{booth}.

\begin{theorem}\label{lindelof} If $X$ is a topological space such that $hL(X)<\mathfrak{t}$, then every sequence of subsets of $X$ has a convergent subsequence.
\end{theorem}

\begin{proof}
	Suppose there is a sequence $(K_n)_{n \in \omega}$ of subsets of
	$X$ without a convergent subsequence. We will construct inductively three
	sequences $(A_\alpha)_{\alpha < \mathfrak{t}}$,	$(p_\alpha)_{\alpha < \mathfrak{t}}$ and $(U_\alpha)_{\alpha <
		\mathfrak{t}}$ with the following properties:
	
	\begin{itemize}
		
		\item[i)] $(A_\alpha)_{\alpha < \mathfrak{t}}$ is a $\subseteq^*$-descending chain of infinite subsets of $\omega$,
		
		\item[ii)] $p_\alpha \in \Ls (K_n)_{n \in A_\alpha}$,
		
		\item[iii)] $U_\alpha$ is an open neighborhood of $p_\alpha$ and $A_{\alpha+1}=\{n \in A_\alpha: U_\alpha \cap K_n=\emptyset \}$.
		
	\end{itemize}
	
	We let $A_0= \omega$. If we already defined $A_\alpha$, the fact that $(K_n)_{n \in A_\alpha}$ does not converge means that there	is a $p_\alpha \in \Ls (K_n)_{n \in A_\alpha} \setminus \Li (K_n)_{n \in A_\alpha}$. Hence there is an open neighborhood	$U_\alpha$ of $p_\alpha$ for which the set $\{n \in A_\alpha : U_\alpha \cap K_n=\emptyset \}$ is infinite and we let	$A_{\alpha+1}$ be this set. For limit ordinals $\gamma < \mathfrak{t}$, since $(A_\alpha)_{\alpha < \gamma}$ is not a tower, we can find $A_\gamma \in [\omega]^\omega$ with $A_\gamma \subseteq^* A_\alpha$ for all $\alpha < \gamma$.
	
	Note that if $\beta>\alpha$ then since $A_\beta \subseteq^*	A_{\alpha+1}$ we have that $p_\beta \in \Ls (K_n)_{n \in A_{\alpha+1}}$ and therefore $p_\beta \notin U_\alpha$. This shows that $(p_\alpha)_{\alpha < \mathfrak{t}}$ is a right-separated	sequence in $X$ and therefore $hL(X) \geq \mathfrak{t}$.

\end{proof}

Note that this gives us an alternative proof of the fact that the Sorgenfrey line has the generalized Bolzano-Weierstrass property (Example \ref{sorgenfrey}), since it is in fact hereditarily Lindel\"{o}f.

The previous theorem can also help us to shed some light on Question \ref{Qds}. Suppose that $X$ is a regular Hausdorff space with $X \notin \gbw$. Then $X$ is not hereditarily Lindel\"{o}f so it contains a right separated subspace $Y$ of type $\omega_1$. But then either $Y$ is an $S$-space or else it contains an uncountable discrete subspace (see \cite{Roi} for a proof of this fact and more information on $S$-spaces). Since it is consistent that there are no $S$-spaces (see \cite{Tod}), we obtain a very partial answer to Question \ref{Qds}:

\begin{theorem}
		It is consistent with ZFC that any $T_3$ topological space $X \not\in \gbw$ contains a discrete subspace of size $\aleph_1$.
\end{theorem}

It is natural to then ask

\begin{question}
		Is it consistent with ZFC that any $T_2$ topological space $X \not\in \gbw$ contains a discrete subspace of size $\aleph_1$?
\end{question}

We don´t know if the number $\mathfrak{t}$ is optimal in Theorem \ref{lindelof}. Note that again discrete spaces (and their one-point compactifications) show that the number $\mathfrak{s}$ is the best we could expect (even for compact Hausdorff spaces), but we don´t know if we can in fact replace $\mathfrak{t}$ by $\mathfrak{s}$, so we ask

\begin{question}
	Is there a space $X \notin \gbw$ such that $hL(X)<\mathfrak{s}$?
\end{question}

Note that necessarily such a space would have to live in a model of $\mathfrak{t}<\mathfrak{s}$ so it cannot be a real example. On the other hand, it is not hard to see that the Fedor\v{c}uk´s space $X$ of Example \ref{fedorchuk} satisfies $hL(X)=\aleph_1=\mathfrak{t}$ (since $\diamondsuit \Rightarrow \mathrm{CH}$), but this doesn´t really show that the number $\mathfrak{t}$ in Theorem \ref{lindelof} is sharp since $X$ is not a ZFC example.

\begin{question}
	Is there, in \emph{ZFC}, a space $X \notin \gbw$ such that $hL(X)=\mathfrak{t}$?
\end{question}

We now give an example that shows that the converse of Theorem \ref{lindelof} is not true. This construction has been used before (see \cite{Haj}) and gives us a space $X \in \gbw$ such that $hL(X)=\mathfrak{s} \geq \mathfrak{t}$.

\begin{example}\label{right}
	Let	$X=\{x_\alpha :\alpha \in \mathfrak{s}\} \subseteq \mathbb{R}$ with the topology $\mathcal{T}$ generated by $\mathcal{T}_{us} \cup \{V_\alpha : \alpha \in \mathfrak{s}\}$, where $\mathcal{T}_{us}$ is the topology inherited by $X$ as a subspace of $\mathbb{R}$ and $V_\alpha=\{x_\xi : \xi \leq
	\alpha\}$. Then any sequence of subsets of $X$ has a convergent subsequence.
\end{example}
\begin{proof}
	First we prove that for any $K \subseteq X$ there is a $\delta \in \mathfrak{s}$ such that $K \cap V_\delta$ is dense in $K$. Since $\mathbb{R}$ is hereditarily separable, $K$ contains a countable set $\{x_{\alpha_n} : n \in \omega\}$ which is dense in $K$ with the topology inherited from $\mathcal{T}_{us}$. Let $\delta=\sup \{\alpha_n
	: n \in \omega\}$ and note that $\delta \in \mathfrak{s}$ because
	$\mathrm{cof}(\mathfrak{s})>\omega$ (see for instance \cite{Dow}). If $U \cap V_\alpha$ with $U \in \mathcal{T}_{us}$	is a basic open set that intersects $K$, then there is an $n \in	\omega$ with $x_{\alpha_n} \in K \cap U$. If $\alpha_n \leq \alpha$	then $x_{\alpha_n} \in (K \cap V_\delta) \cap (U \cap V_\alpha)$ and if $\alpha \leq \alpha_n$ then $(K \cap V_\delta) \cap (U \cap V_\alpha) = K \cap (U \cap V_\alpha) \neq \emptyset$. In either case $U \cap V_\alpha$ intersects $K \cap V_\delta$ and therefore $K \cap V_\delta$ is dense in $K$.
	
	Fix a sequence $(K_n)_{n \in \omega}$ of subsets of $X$. By the
	previous paragraph and the fact that replacing each $K_n$ by one
	of its dense subsets does not change the $\Ls$ or the $\Li$ of any
	subsequence of $(K_n)_{n \in \omega}$, we may assume that each $K_n$ is contained in $V_{\delta_n}$ for some $\delta_n \in \mathfrak{s}$. Let $\gamma=\sup \{\delta_n : n \in \omega\}$ and note that $\gamma \in \mathfrak{s}$ using again that $\mathrm{cof}(\mathfrak{s})>\omega$.
	
	Consider the topology $\mathcal{T}_\gamma$ in $X$ generated by $\mathcal{T}_{us} \cup \{V_\alpha : \alpha \in \gamma\}$. Given that $K_n=K_n \cap V_\alpha$ for any $\alpha \geq \gamma$, it is clear that for
	computing the $\Ls$ or the $\Li$ of any subsequence of $(K_n)_{n \in \omega}$ it does not make any difference if we use the topology $\mathcal{T}_\gamma$ or the topology $\mathcal{T}$. Since $\mathcal{T}_\gamma$ has weight less than $\mathfrak{s}$, Theorem \ref{booth} tells us that $(K_n)_{n \in \omega}$ has a convergent subsequence in $(X,\mathcal{T}_\gamma)$ and	therefore in $(X,\mathcal{T})$.
	
\end{proof}

We finish this section with a curious observation which is some sort of converse of Theorem \ref{lindelof}.

\begin{theorem}[Assume $\mathfrak{s}=\aleph_1$]
	If $X$ is a regular Hausdorff space in $\gbw$ then either $hL(X)<\mathfrak{t}$ or else $X$ contains an $S$-space.
\end{theorem}
\begin{proof}
	If $hL(X)\geq \mathfrak{t}=\aleph_1$ then $X$ contains a right-separated subspace $Y$ of type $\omega_1$. Since $X \in \gbw$, by Theorem \ref{discrete}, $X$ (and therefore $Y$) contains no discrete subset of size $\mathfrak{s}$ (and hence no uncountable discrete subset) so $Y$ is an S-space.
\end{proof}

\section{Two discouraging examples}

The goal of this final section is to exhibit (assuming $\diamondsuit$) two spaces $X$ and $Y$ for which all of the seemingly relevant cardinal functions coincide and yet $X \in \gbw$ and $Y \notin \gbw$. Since we are assuming CH we have, at the same time, all cardinal invariants of the continuum collapsed into one. This seems to show, to some extent, that it is hopeless to try to characterize the generalized Bolzano-Weierstrass property for general topological spaces in a similar fashion as we did for metrizable spaces. Of course, there are many cardinal functions that we haven't explored, and there are other kind of restrictions on the spaces that one could also investigate (normality, compactness, homogeneity, etc.). 

Before looking at the examples, lets remember that, given a topological space $(Z,\mathcal{T})$ and an elementary submodel $M$ of (a large enough initial fragment of) the universe containing $(Z,\mathcal{T})$, one can consider another topological space $Z_M$ whose underlying set is $Z \cap M$ and whose topology has the collection $\mathcal{T}\!\upharpoonright\! M =\{U \cap M : U \in \mathcal{T}\cap M \}$ as a base (see \cite{Jun} for more details on this construction and many of its properties).

\begin{theorem}
	If the topological space $(Z,\mathcal{T})$ is not in $\gbw$ and $M$ is countably closed then $Z_M$ is not in $\gbw$.
\end{theorem}

\begin{proof}
	By elementarity we can fix $(K_n)_{n \in \omega} \in M$ a sequence of subsets of $Z$ without a convergent subsequence. We show that $(K_n\cap M)_{n \in \omega}$ has no convergent subsequence in $Z_M$. Given $A \in [\omega]^\omega$, since $A \in M$ (because $M$ is countably closed), by elementarity there is a $z \in Z \cap M$ such that $z \in \Ls (K_n)_{n\in A} \setminus \Li (K_n)_{n \in A}$. In particular (again by elementarity) there is a neighborhood $V \in \mathcal{T}\cap M$ of $z$ such that $V \cap K_n=\emptyset$ for infinitely many $n \in A$.
	
	Note that for any $U \in \mathcal{T} \cap M$ and $n \in \omega$ we have, using elementarity one more time, that $U \cap K_n \neq \emptyset \Leftrightarrow (U \cap M)\cap (K_n \cap M) \neq \emptyset$. This clearly implies that $z \in \Ls_{Z_M} (K_n\cap M)_{n \in A}$ and also, since $(V\cap M)\cap(K_n\cap M)=\emptyset$ for infinitely many $n$'s, that  $z \notin \Li_{Z_M} (K_n\cap M)_{n \in A}$. Therefore $(K_n \cap M)_{n \in A}$ does not converge in $Z_M$.
\end{proof}

Now we are ready for our examples. From now on assume $\diamondsuit$. 

The first space is the disjoint union of two spaces $X=X_1 \cup X_2$ where $X_1$ is the space of Example \ref{right}, just that now we are assuming $\mathfrak{s}=\aleph_1$, and $X_2$ is any countable, not first countable space. We saw that $X_1 \in \gbw$ and from Theorem \ref{lindelof} we have that $X_2 \in \gbw$ since $X_2$ is countable and hence hereditarily Lindel\"{o}f. It is straightforward to see that the disjoint union of two spaces in $\gbw$ is again in $\gbw$ so we get $X \in \gbw$.

For the second example we start with the Fedor\v{c}uk´s space of Example \ref{fedorchuk}, lets call it $(Z,\mathcal{T})$. Then we take a countably closed elementary submodel $M$ of size $2^{\aleph_0}=\aleph_1$ with $(Z,\mathcal{T}) \in M$. Finally we let $Y=Z_M$ as defined before. By the previous theorem, since $Z \notin \gbw$ and $M^\omega\subseteq M$, we have that $Y \notin \gbw$.

Clearly $|X|=|Y|=\aleph_1$. We know from Example \ref{right} that $hL(X) \geq \mathfrak{s}=\aleph_1$ and when proving that $X_1 \in \gbw$ we showed that any subset of $X_1$ has a dense subset of size less than $\mathfrak{s}$, so here we get that $X_1$ (and hence $X$) is hereditarily separable. Since $X_2$ is not first countable we have $\chi (X)=\aleph_1$. From $Y \notin \gbw$ and Theorem \ref{lindelof} we get $hL(Y) \geq \mathfrak{t}=\aleph_1$. Since $Z$ is hereditarily separable we have that $Y$ is also hereditarily separable (see \cite[Theorem 4.1]{Jun}). Also $\chi(Z)=\aleph_1$ because $Z$ has no non-trivial convergent sequences. Thus we have $$|X|=|Y|=w(X)=w(Y)=\chi(X)=\chi(Y)=hL(X)=hL(Y)=\aleph_1=\mathfrak{s}=\mathfrak{t}$$ and $$hd(X)=hd(Y)=s(X)=s(Y)=\aleph_0.$$

\bibliographystyle{plain}

\end{document}